\documentclass[a4paper]{amsart}
\usepackage[utf8]{inputenc}
\usepackage[T2A]{fontenc}
\usepackage[english]{babel}

\usepackage{amssymb}
\usepackage{amsthm}

\usepackage{tikz}
\usetikzlibrary{shapes,snakes}

\newtheorem{lemma}{Lemma}
\newtheorem{theorem}{Theorem}
\newtheorem*{theoremT}{Theorem T}

\def\hm#1{#1\nobreak\discretionary{}{\hbox{\ensuremath{#1}}}{}}

\def\Z{{\mathbb Z}}
\def\N{{\mathbb N}}
\def\R{{\mathbb R}}

\def\taue{\tau^{(e)}}

\def\tt{{\mathfrak t}}

\def\eps{\varepsilon}

\def\res{\mathop{\mathrm{res}}}

\def\le{\leqslant}
\def\ge{\geqslant}

\def\dd{\mathrm{d}}
\def\OO{\mathrm{O}}
\def\oo{\mathrm{o}}

\newcommand*\circled[1]{\tikz[baseline=(char.base)]{
    \node[shape=circle,draw,inner sep=0.5pt] (char) {$#1$};}}
\newcommand*\ellipsed[1]{\tikz[baseline=(char.base)]{
    \node[shape=rectangle,draw,inner sep=1pt,rounded corners=1ex] (char) {$#1$};}}
\newcommand*\circleds[1]{\tikz[baseline=(char.base)]{
    \node[shape=circle,draw,inner sep=0.2pt] (char) {$#1$};}}
\newcommand*\ellipseds[1]{\tikz[baseline=(char.base)]{
    \node[shape=rectangle,draw,inner sep=1.5pt,rounded corners=1ex] (char) {$#1$};}}
\def\m{{\circled{m}}}
\def\mm{{\ellipsed{m-1}}}
\def\ms{{\circleds{\scriptstyle m}}}
\def\mms{{\ellipseds{\scriptstyle m-1}}}

\usepackage[unicode]{hyperref}

\begin{document}

\title{Exponential divisor functions}
\author{Andrew V. Lelechenko}
\address{I.~I.~Mechnikov Odessa National University}
\email{1@dxdy.ru}

\keywords{Exponential divisor function, generalized divisor function, average order}
\subjclass[2010]{
11N37, 
11M06, 
11A25  
}

\begin{abstract}
Consider the operator~$E$ on arithmetic functions such that~$Ef$ is the multiplicative arithmetic function defined by~$(Ef)(p^a) \hm= f(a)$ for every prime power $p^a$. We investigate the behaviour of~$E^m\tau_k$, where $\tau_k$ is a $k$-dimensional divisor function and~$E^m$ stands for the $m$-fold iterate of~$E$. We estimate the error terms of~$\sum_{n\le x} E^m\tau_k(n)$ for various combinations of $m$ and~$k$. We also study properties of~$E^mf$ for arbitrary $f$ and sufficiently large~$m$.

Our study provides a unified approach to functions with exponential divisors. We improve special cases of the Dirichlet asymmetric divisor problem and several results on the exponential divisor and totient functions.
\end{abstract}

\maketitle

\section{Introduction}

Consider the set $\mathcal A$ of arithmetic functions, the set $\mathcal M_{PI}$ of multiplicative prime-in\-de\-pen\-dent functions  and the operator $E\colon {\mathcal A} \to {\mathcal M}_{PI}$ such that
$$
(Ef)(p^a) = f(a)
$$
for every prime power $p^a$.
The behaviour of $Ef$ for various special cases of $f$ has been widely studied, starting with the pioneering paper of Subbarao~\cite{subbarao1972} on~$E\tau$ and~$E\mu$, where $\tau$ is the divisor function and $\mu$ is the Möbius function. The most notable of them are the paper of Wu~\cite{wu1995} on~$E\tau$, the paper of Pétermann and Wu~\cite{petermann1997} on the sum of $E$-divisors, the papers of Tóth~\cite{toth2004,toth2007b,toth2007a} on certain $E$-functions, the paper of Pétermann \cite{petermann2010} on $E\phi$, the paper of Cao and Zhai \cite{cao2010} on estimates of certain $E$-functions under Riemann hypothesis (RH).

The established by previous authors notation for $Ef$ is $f^{(e)}$. We write~$E^2f(n)$, $E^3f(n)$\dots\ meaning~$(E^2f)(n)$, $(E^3f)(n)$\dots

The primary aim of the current paper is to investigate effects of essentially multiple applications of operator $E$ on different functions, but we also obtain several important results in the case of the single application.

We improve previously known error terms for the special cases of Dirichlet asymmetric divisor problem (namely, for~$\tau_{1,2^r}$ in Lemma~\ref{l:sum-of-tau1m-precise} and for~$\tau_{1,2^r,2^r}$ in Lemma~\ref{l:sum-of-tau1mm-precise}).
Section \ref{s:theorem-t} is dedicated to the refinement of Tóth's theorem~\cite{toth2007a}.

Mostly we consider divisor functions, but in the last section a generalization over all arithmetic functions is discussed.

\section{Notation}

In asymptotic relations we use
Landau symbols $\Omega$, $\OO$ and $\oo$, Vinogradov symbols~$\ll$ and $\gg$ in their usual meanings. All asymptotic relations are given as an argument (usually $x$) tends to $+\infty$.

Letter $p$ with or without indexes denotes a rational prime.

As usual $\zeta(s)$ is the Riemann zeta-function. For complex $s$ we denote~$\sigma:=\Re s$ and~$t:=\Im s$.

Letter $\gamma$ denotes the Euler–Mascheroni constant, $\gamma \approx 0.577$.

Everywhere $\eps>0$ is an arbitrarily small number (not always the same even in one equation).

We write $f\star g$ for the Dirichlet convolution:
$ (f \star g)(n) = \sum_{d \mid n} f(d) g(n/d) $.

\smallskip

Let $\tau$ be the divisor function, $\tau(n) = \sum_{d|n} 1$.
Denote
$$
\tau(a_1,\ldots,a_k; n) = \sum_{d_1^{a_1}\cdots d_k^{a_k} = n} 1
$$
and $\tau_k = \tau(\underbrace{1,\ldots,1}_{k\text{~times}}; \cdot)$.
Then $\tau \equiv \tau_2 \equiv \tau(1,1; \cdot)$.

Now let $\Delta(a_1,\ldots,a_k; x)$ be an error term in the asymptotic estimate of the sum~$\sum_{n\le x} \tau(a_1,\ldots,a_k; n)$. (See \cite{kratzel1988} for the form of the main term.) For the sake of brevity denote $
\Delta_k(x) = \Delta(\underbrace{1,\ldots,1}_{k\text{~times}}; x)
$.

Finally, $\theta(a_1,\ldots,a_k)$ denotes throughout the paper a real value such that
$$ \Delta(a_1,\ldots,a_k; x) \ll x^{\theta(a_1,\ldots,a_k)+\eps}. $$

We abbreviate $\theta_k$ for the exponent of $x$ in $\Delta_k(x)$.

\section{Elementary properties of $E$}

One can easily check that
$$
E(f+g) = Ef+Eg
\qquad \text{and} \qquad
E(f\cdot g) = Ef \cdot Eg.
$$

Define (following Subbarao~\cite{subbarao1972}) exponential  convolution $\odot$ as
$$
(f \odot g) (p^n) = \sum_{d\mid n} f(p^d) g(p^{n/d}).
$$

Then
$
E(f\star g) = Ef \odot Eg
$,
because
\begin{multline*}
E(f\star g) (p^n)
= (f\star g)(n)
= \sum_{d\mid n} f(d) g(n/d)
= \\
= \sum_{d\mid n} Ef(p^d) Eg(p^{n/d})
= (Ef \odot Eg)(p^n).
\end{multline*}



\begin{lemma}
The only solution of $Ef = f$ is $f(n) = 1$ for all $n$.
\end{lemma}
\begin{proof}
Assume the contrary: there is $n$ such that $f(n)\ne1$. Choose the least of all such $n$.

This number cannot have more than one prime divisor: if $n=\prod_i p_i^{a_i}$ and $f(n)\hm\ne1$ then for some index $i$ we have $f(p_i^{a_i}) \ne 1$ and $n$ is not the least. Now $n=p^a$. Since $f$ is prime-independent $f(2^a) = f(n) \ne 1$, so in fact $n=2^a$.

The requirement $Ef = f$ induces~$f(n)=f(a)$. But $a<n$; it is a~contradiction.
\end{proof}

\begin{lemma}
For every arithmetic function $f$ such that $f(n) \ll n^a$ for some constant~$a$ we have
$ Ef(n) \ll n^\eps $.
\end{lemma}
\begin{proof}
By result of Suryanarayana and Sita Rama Chandra Rao \cite{suryanarayana1975} we have
$$
\limsup_{n\to\infty} {\log Ef(n) \log\log n \over \log n} = \sup_n {\log f(n) \over n}.
$$
But as soon as $f(n) \ll n^a$ the right side is bounded and so
$$
{\log Ef(n) \over \log n} \to 0 \quad \text{as~} n\to\infty.
$$
It means that $Ef(n) \ll n^\eps$.
\end{proof}

\section{Multiexponential divisor functions}

Consider the family of functions
$$
\tau^{(e)}, E^2\tau, E^3\tau \ldots
$$

Define $A_m := \{ n \mid E^m\tau(n) \ne 1 \}$. Then
$$
A_0 = \N \setminus \{1\},
\qquad
A_1 = \{ n \mid \mu(n)=0 \} = \{ 4,8,9,12\ldots \}
$$
$$
A_m = \bigcup_{p} \bigcup_{a\in A_{m-1}} p^a R_p,
\quad \text{where} \quad
R_p := \{ n \mid \gcd(n,p)=1 \}
$$

Which are the lowest elements of these infinite sets? Denote
$$
\m = \min A_m,
\qquad
\m' = \min (A_m \setminus \{\m\}),
\qquad
\m'' = \min (A_m \setminus \{\m, \m'\}).
$$
Then $\circled0 = 2$, $\circled1 = 4$, $\circled2 = 16$, $\circled3 = 65\,536$ and generally $\m = 2^{\mms}$.

The set $A_m$ contains not only $\m$, but also numbers of form $(2k+1)\m$ for each integer $k$. Other elements of $A_m$ must be divisible by $n=p^a$, where either $p>2$ or $a>\mm$. Let
$$
B_m := A_m \setminus (2\N+1) \m.
$$

\begin{lemma}
For integer $m>1$ we have
\begin{equation}\label{eq:m'-m''}
\m' = 3\m,
\qquad
\m'' = 5\m
\end{equation}
\end{lemma}
\begin{proof}
We prove only the first statement; the second can be proven in the same way.

It is enough to show that $\min B_m > 3\m$; let us prove this by induction. From the discussion above we have
$$
\min B_m \ge \min\{ 3^{\mms}, 2^{\mms'} \}.
$$
For $m=2$ one can evaluate
$$
\min B_2 \ge \min\{ 3^4, 2^8 \} = 81 > 3 \cdot \circled2 = 3\cdot16.
$$
Now let $m>1$, $\m' = 3\m$. Then
$$
\min B_{m+1} \ge \min \{ 3^{\ms}, 2^{3\ms} \}
= 3^{\ms} > 3\cdot 2^{\ms} = 3 \ellipsed{m+1}.
$$
\end{proof}

The asymptotic behaviour of $\taue$ has been studied in details by Wu and Smati in \cite{wu1995} and \cite{smati1997}. Now we are going to study $E^m\tau$ for $m>1$.

\begin{lemma}
For a fixed integer $m>1$ we have
\begin{equation}\label{eq:Emtau-zeta-expansion}
\sum_{n=1}^\infty {E^m\tau(n) \over n^s}
= {
	\zeta(s) \zeta(\m s) \zeta\bigl((2\m+1)s\bigr) \zeta(3\m s)
	\over
	\zeta\bigl( (\m+1) s \bigr) \zeta(2\m s)
	} H(s),
\end{equation}
where $H(s)$ converges absolutely for
$
\sigma > {1/ (3\m+1)}
$.
\end{lemma}
\begin{proof}
Taking into account \eqref{eq:m'-m''} and
$E^m\tau(\m) = E^m\tau(3\m) = 2$, then
$$
S(x):= \sum_{n\ge0} E^m\tau(p^n) x^n
= \sum_{n\ge0} x^n + x^\ms + x^{3\ms} + \OO(x^{5\ms}),
$$
\begin{align*}
(1-x)S(x) &= 1+x^\ms-x^{\ms+1}+x^{3\ms}+\OO(x^{3\ms+1}),
\\
(1-x)(1-x^m)S(x) &= 1 - x^{\ms+1} - x^{2\ms} + x^{2\ms+1} + x^{3\ms} +\OO(x^{3\ms+1})
\\
{(1-x)(1-x^\ms) \over 1-x^{\ms+1}} S(x)
&= 1-x^{2\ms}+x^{2\ms+1}+x^{3\ms} +\OO(x^{3\ms+1})
\\
{(1-x)(1-x^\ms) \over (1-x^{\ms+1})(1-x^{2\ms})} S(x)
&= 1+x^{2\ms+1}+x^{3\ms} +\OO(x^{3\ms+1})
\end{align*}
So
$$
{(1-x)(1-x^\ms)(1-x^{2\ms+1})(1-x^{3\ms}) \over (1-x^{\ms+1})(1-x^{2\ms})} S(x)
= 1+\OO(x^{3\ms+1})
$$
and this implies \eqref{eq:Emtau-zeta-expansion}.
\end{proof}

\begin{theorem}
For a fixed integer $m>1$ we have
\begin{equation}\label{eq:sum-of-Emtau}
\sum_{n\le x} E^m\tau(n) = A_m x + B_m x^{1/\ms} + \OO(x^{\alpha_m}),
\quad
\alpha_m = 1/(\m+1),
\end{equation}
where $A_m$ and $B_m$ are computable constants.
\end{theorem}
\begin{proof}
Follows from \eqref{eq:Emtau-zeta-expansion}, a classic estimate from the book of Krätzel~\cite[Th.~5.1]{kratzel1988}
\begin{equation}\label{eq:tau1m-estimate}
\sum_{n\le x} \tau(1, \m; n)
= \zeta(\m) x + \zeta(1/\m) x^{1/\ms} + \OO(x^{\theta(1,\ms)}),
\quad
\theta(1,\m) = 1/(\m+1)
\end{equation}
and the convolution method.
\end{proof}

We can improve the error term in \eqref{eq:sum-of-Emtau} under RH. First of all we should estimate~\eqref{eq:tau1m-estimate} more precisely. Using an exponent pair $A^{\mms-1}B(0,1)$, where~$A$ and~$B$ are van der Corput's processes, one can obtain $\theta(1,\m) \hm= 1/(\m\hm+\mm)$. The following lemma provides even better result by sophisticated selection of the exponent pair.

\begin{lemma}\label{l:sum-of-tau1m-precise}
For a fixed integer $r\ge5$ we have
$$
\theta(1,2^r) = {2^r-2r \over 2^{2r} - r\cdot2^r - 2r^2 + 2r - 4}
< {1\over 2^r+r}.
$$
\end{lemma}
\begin{proof}
Consider an exponent pair
$$
(k_r, l_r) := A^{r-1} B A^{r-3} BAB (0,1).
$$
To evaluate $k_r$ and $l_r$ we map exponent pairs into the real projective space (the concept of such mapping traces back to Graham \cite{graham1986}):
$$
\R^2 \ni (k,l) \mapsto (k:l:1) \in \R^3/(\R\setminus\{0\}).
$$
We have
$$
A(k,l) \mapsto {\mathcal A} (k:l:1),
\qquad
{\mathcal A} = \begin{pmatrix} 1&0&0 \\ 1&1&1 \\ 2&0&2 \end{pmatrix} = {\mathcal S} \begin{pmatrix} 1&1&0 \\ 0&1&0 \\ 0&0&2 \end{pmatrix} {\mathcal S}^{-1},
$$
where
$
{\mathcal S} = \left(\begin{smallmatrix} 0&-1&0 \\ 1&0&1 \\ 0&2&1  \end{smallmatrix}\right)
$,
and
$$
B(k,l) \mapsto {\mathcal B}(k:l:1),
\qquad
{\mathcal B} = \begin{pmatrix} 0&2&-1 \\ 2&0&1 \\ 0&0&2 \end{pmatrix}.
$$
Thus
$
{\mathcal A}^n = {\mathcal S} \left( \begin{smallmatrix} 1&n&0 \\ 0&1&0 \\ 0&0&2^n \end{smallmatrix} \right) {\mathcal S}^{-1}
$
and
$$
{\mathcal A}^{r-1} {\mathcal B} {\mathcal A}^{r-3} {\mathcal B} {\mathcal A} {\mathcal B} (0:1:1)
=
\begin{pmatrix}
4\cdot 2^r - 8r \\
8\cdot 2^{2r} - (12r+16) \cdot 2^r + 8r^2+8r+16 \\
8\cdot 2^{2r} - (8r+16)  \cdot 2^r + 16r
\end{pmatrix}.
$$
Returning to $\R^2$ we get
$$
k_r = {2^r-2r \over 2\cdot 2^{2r}-(2r+4)\cdot 2^r + 4 r},
\quad
l_r = 1 - { r\cdot 2^r - 2r^2+2r-4 \over 2\cdot 2^{2r}-(2r+4)\cdot 2^r + 4 r}.
$$
Now for $r\ge5$
$$
2l_r - 2\cdot2^r k_r - 1 = {2r^2+4-2\cdot 2^r \over 2^{2r}-(r+2)\cdot 2^r + 2r} < 0.
$$
This proves that $(k_r, l_r)$ satisfies the second case of \cite[Th. 5.11]{kratzel1988} and finally
$$
\theta(1,2^r) = {k_r \over 2^rk_r - l_r + 1}.
$$
\end{proof}

In special cases the value of $\theta(1,\m)$ in \eqref{eq:sum-of-Emtau} can be estimated even more precisely. The case of $m=0$ is classic, the best modern result is by Graham and Kolesnik \cite{graham1988}. For bigger $m$ we calculated estimates using \cite[Th.~5.11, Th.~5.12]{kratzel1988}, selecting appropriate exponent pair using the method described in~\cite{lelechenko2013-acta}; see Table \ref{tbl:tau1m-precise}.

\begin{table}[h]
\begin{tabular}{@{}r|r|r|c@{}}
$m$ & $\m$ & $\theta(1,\m)$  & Exp. pair or ref.\\\hline
0 &  2  & $1057/4785+\eps               \approx 0.220899$  & Graham and Kolesnik \cite{graham1988} \\\hline
1 &  4  & $1448/10331+\eps               \approx 0.140161$  & $A H_{05}$ \\ \hline
2 & 16  & $15/307   \approx 0.048860$ &  $A^3BA^2BA^4B I$ \\
\end{tabular}
\smallskip
\caption{Special values of $\theta(1,\cdot)$. Exponent pairs are written in terms of  $A$- and $B$-processes. Here~$I=(0,1)$ and~$H_{05}=(32/205\hm+\eps, 269/410\hm+\eps)$ is an exponent pair from~\cite{huxley2005}.}
\label{tbl:tau1m-precise}
\end{table}

\begin{theorem}\label{th:Emtau-under-RH}
Under RH for a fixed $m>1$
\begin{equation}\label{eq:sum-of-Emtau-under-RH-exponent}
\alpha_m  = {1-\theta(1,\m) \over \m+2-2(\m+1)\theta(1,\m)},
\end{equation}
where $\alpha_m$ is an exponent from \eqref{eq:sum-of-Emtau}.
Note that if $\theta(1,\m) < 1/(\m+1)$ then~$\alpha_m \hm< 1/(\m+1)$ too.
\end{theorem}
\begin{proof}
By \eqref{eq:Emtau-zeta-expansion} we get
$$
E^m\tau = \tau(1,\m; \cdot) \star \mu_{\ms+1} \star h,
$$
where $\mu_k(n^k) = \mu(n)$ and $\mu_k(n)=0$ in other cases and $h(n)$ is the $n^{\text{th}}$ coefficient in Dirichlet series
$$
{
	\zeta\bigl((2\m+1)s\bigr) \zeta(3\m s)
	\over
	\zeta(2\m s)
	} H(s).
$$
We see that $\sum_{n\le x} h(n) n^{-s}$ converges absolutely for $\sigma > 1/2\m$.

Nowak (see \cite[Th. 2]{nowak1993} or Cao and Zhai \cite[Th. 1, Cor. 1.1]{cao2013}) proved that if~$a\hm\le b<c<2(a+b)$ and $\theta(a,b) < 1/c$, then under RH for~$f\hm=\tau(a,b; \cdot) \star \mu_c$ one have
\begin{equation}\label{eq:nowak-result}
\sum_{n\le x} f(n)
= {\zeta(b/a) \over \zeta(c/a)} x^{1/a}
+ {\zeta(a/b) \over \zeta(c/b)} x^{1/b}
+ \OO(x^{\beta+\eps}),
\end{equation}
where
$$
\beta = {1-a\theta(a,b) \over a+c-2ac\theta(a,b)}.
$$
Substitution $a=1$, $b=\m$, $c=\m+1$ gives us \eqref{eq:sum-of-Emtau-under-RH-exponent}.

The rest of the theorem can be proven by convolution method.
\end{proof}

We can also estimate the error term from the bottom.

\begin{theorem}
For a fixed integer $m>1$ we have
$$
\sum_{n\le x} E^m\tau(n) = A_m x + B_m x^{1/\ms} + \Omega(x^{1/2(\ms+1)}).
$$
\end{theorem}
\begin{proof}
Follows from \eqref{eq:Emtau-zeta-expansion}, which shows that Dirichlet series
$$\sum_{n=1}^\infty {E^m\tau(n) \over n^s}$$
has $\zeta\bigl( (\m+1) s \bigr)$ in the denominator and thus
has infinitely many poles at line~$\sigma\hm=1/2(\m+1)$.
\end{proof}

\section{Exponential multidimensional divisor functions}\label{s:theorem-t}

Consider a family of functions
$$
\tau^{(e)}, \tau_3^{(e)}, \tau_4^{(e)}\ldots
$$

In \cite{toth2007a} Tóth has proved the following general result.
\begin{theoremT}
Let $f$ be a complex valued multiplicative
arithmetic function such that

a) $f(p)=f(p^2)=\cdots =f(p^{\ell-1})=1$,
$f(p^{\ell})=f(p^{\ell+1})=k$ for every prime~$p$, where $\ell, k\ge
2$ are fixed integers and

b) there exist constants $C, m>0$ such that $|f(p^a)|\le Ca^m$ for
every prime~$p$ and every $a\ge \ell +2$.

Then for $s\in \mathbb{C}$

i) $$ F(s):= \sum_{n=1}^{\infty} \frac{f(n)}{n^s} =
\zeta(s)\zeta^{k-1}(\ell s) V(s), \qquad \Re s \hm> 1,
$$
where the Dirichlet series $V(s):=\sum_{n=1}^{\infty} \frac{v(n)}{n^s}$ is absolutely
convergent for~$\Re s > \frac1{\ell+2}$,

ii) $$\sum_{n\le x} f(n)= C_fx + x^{1/\ell} P_{f,k-2}(\log x)
+\OO(x^{u_{k,\ell}+\varepsilon}),
$$
for every $\varepsilon >0$, where $P_{f,k-2}$ is a polynomial of
degree $k-2$, $u_{k,\ell} \hm{:=} \frac{2k-1}{3+(2k-1)\ell}$ and
$$
C_f: =\prod_p \left(1 + \sum_{a=\ell}^{\infty}
\frac{f(p^a)-f(p^{a-1})}{p^a} \right).
$$

iii) The error term can be improved for certain values of $k$ and
$\ell$. For example in the case $k=3$, $\ell=2$ it is
$\OO(x^{8/25}\log^3 x)$.
\end{theoremT}

Tóth considered $\tau_k^{(e)}$, showed that
\begin{equation}\label{eq:Etauk-zeta-expansion}
\sum_{n=1}^\infty {\tau_k^{(e)}(n) \over n^s}
= \zeta(s) \zeta^{k-1}(2s) H(s),
\end{equation}
where $H(s)$ converges absolutely for $\sigma > 1/5$, and thus obtained that
\begin{equation}\label{eq:sum-of-Etauk-Toth}
\sum_{n\le x} \tau_k^{(e)}(n)
= C_k x + x^{1/2} S_{k-2}(\log x) + \OO(x^{w_k+\eps}),
\end{equation}
where $w_k = (2k-1)/(4k+1)$.

We are going to improve this estimate by improving Theorem T.

\begin{theorem}\label{th:toth-improved}
In conditions and notations of Theorem T
we can take
\begin{equation}\label{eq:ukl-improved}
u_{k,\ell} = {1\over \ell+1-\theta_{k-1}}.
\end{equation}
In particular for $k\geqslant5$ we have
\begin{equation}\label{eq:ukl-inequation}
u_{k,\ell} \le {k+1 \over 3 + (k+1) \ell}
< {2k-1 \over 3 + (2k-1) \ell}.
\end{equation}
and for $k\ge1000$ an inequality
\begin{equation}\label{eq:ukl-inequation-big-k}
u_{k,\ell} \le {1\over l+ck^{-2/3}},
\qquad
c \le \left( 40 \over 267 \right)^{2/3} .
\end{equation}
holds.
\end{theorem}
\begin{proof}
The analysis of the proof of the Theorem T shows that the error term in it is caused by the estimate of $$\Delta(1,\underbrace{\ell,\ldots,\ell}_{k-1\text{~times}}).$$
Let us show how this estimate can be improved.

Taking into account
$$
\sum_{n\le x} \tau(\underbrace{\ell,\ldots,\ell}_{k\text{~times}}; n)
= \sum_{n\le x^{1/\ell}} \tau_k(n),
$$
we obtain that
$$
\Delta(\underbrace{\ell,\ldots,\ell}_{k\text{~times}}; x)
\ll x^{\theta_k/\ell + \eps}
\quad\text{and}\quad
\Delta(\underbrace{\ell,\ldots,\ell}_{k-1\text{~times}}; x)
\ll x^{\theta_{k-1}/\ell + \eps}.
$$

Substituting $p=k$, $q=1$, $(a_1,\ldots,a_p) = (1,\ell,\ldots,\ell)$ into \cite[Th.~6.8]{kratzel1988} we get
$$
\Delta(1,\underbrace{\ell,\ldots,\ell}_{k-1\text{~times}}; x)
\ll x^{u_{k,l} + \eps},
$$
where $u_{k,l} $ can be chosen as in \eqref{eq:ukl-improved}.

Values of $\theta_k$ have been widely studied: Huxley \cite{huxley2005} proved that $\theta_2 \hm\le 131/416$; Kolesnik~\cite{kolesnik1981} proved $\theta_3 \le 43/96$; many other special cases can be found in the book of Titchmarsh~\cite[Ch. 12]{titchmarsh1986} and in the paper of Ivić and  Ouellet~\cite{ivic1989}. Namely, \cite[Th.~12.3]{titchmarsh1986} gives an estimate
$$
\theta_k \le {k-1\over k+2}, \qquad k\ge4,
$$
which implies inequality \eqref{eq:ukl-inequation} and \cite[Th. 1]{ivic1989} together with Ford's estimate \cite{ford2002} and Kolpakova's estimate \cite{kolpakova2011} implies \eqref{eq:ukl-inequation-big-k}.
\end{proof}

Substitution $l=2$ and $k=k$ into Theorem \ref{th:toth-improved} allows to decrease $w_k$ in \eqref{eq:sum-of-Etauk-Toth} to
$$
w_k = {1\over 3-\theta_{k-1}}.
$$

Note that the application of Theorem \ref{th:toth-improved} instead of Theorem T also improves all other results of \cite{toth2007a}. Namely,
$$
\sum_{n\le x} \bigl( \tau^{(e)}(n) \bigr)^r
= A_r x + x^{1/2} P_{2^r-2}(\log x) + \OO(x^{u_r+\eps}),
\qquad
u_r = {2^r+1 \over 2^{r+1}+5},
$$
$$
\sum_{n\le x} \bigl( \phi^{(e)}(n) \bigr)^r
= B_r x + x^{1/3} R_{2^r-2}(\log x) + \OO(x^{t_r+\eps}),
\qquad
t_r = {2^r+1 \over 3(2^r+2)}.
$$

\begin{theorem}
For a fixed integer $k>1$ we have
$$
\sum_{n\le x} \tau_k^{(e)}(n)
= C_k x + x^{1/2} S_{k-2}(\log x) +  \Omega(x^{1/(4-2/k)}).
$$
\end{theorem}
\begin{proof}
Application of Kühleitner and Nowak \cite[Th.~2]{kuhleitner1994} with
$f_1=\cdots=f_k=\zeta$, $m_1=1$, $m_2=\cdots=m_k=2$
gives us
$$
\alpha = {k\over 2\bigl(1+2(k-1)\bigr)} = {1\over 4-2/k}.
$$
\end{proof}

\section{Multiexponential multidimensional divisor functions}

Consider the two-parametric family of functions
$$
\{ E^m\tau_k \}, \qquad k>2, \quad m>1.
$$

One can check that for $m>1$ the least arguments on which $E^m\tau_k(n) \hm\ne 1$ are (as in the case of $k=2$) $\m$, $3\m$ and $5\m$.

\begin{lemma}
For fixed integer $k>2$, $m>1$ we have
\begin{equation}\label{eq:Emtauk-zeta-expansion}
\sum_{n=1}^\infty {E^m\tau_k(n) \over n^s} = {\zeta(s) \zeta^{k-1}(\m s) \over \zeta^{k-1}\bigl((\m+1)s\bigr) \zeta^{k(k-1)/2}(2\m s) } H(s),
\end{equation}
where $H(s)$ converges absolutely for
$
\sigma > {1/( 2\m+1)}
$.
\end{lemma}
\begin{proof}
As soon as $E^m\tau_k(\m)=\tau_k(2)=k$ we get
$$
S(x):= \sum_{n\ge0} E^m\tau_k(p^n)
= \sum_{n\ge0}x^n + (k-1)x^\ms + \OO(x^{2\ms+1}),
$$
\begin{align*}
(1-x)S(x) &= 1+(k-1)x^\ms - (k-1)x^{\ms+1} +  \OO(x^{2\ms+1}), \\
(1-x)(1-x^\ms)^{k-1}S(x) &= 1 - (k-1)x^{\ms+1} - {k(k-1)\over2} x^{2\ms} + \OO(x^{2\ms+1}),
\end{align*}
which implies lemma's statement.
\end{proof}

\begin{theorem}\label{th:Emtauk-asymp}
For fixed integer $k>2$, $m>1$ we have
\begin{equation}\label{eq:Emtauk-estimate}
\sum_{n\le x} E^m\tau_k(n)
= K_{m,k} x + x^{1/\ms} R_{m,k-2}(\log x) + \OO(x^{\alpha_{m,k}+\eps}),
\end{equation}
where $K_{m,k}$ is a computable constant, $R_{m,k-2}$ is a polynomial of degree~$k-2$ and
$$
\alpha_{m,k} = {1\over \m+1-\theta_{k-1}}.
$$
\end{theorem}
\begin{proof}
Follows from \eqref{eq:Emtauk-zeta-expansion} and Theorem \ref{th:toth-improved}.
\end{proof}

\begin{theorem}
For fixed integer $k>2$, $m>1$ we have
$$
\sum_{n\le x} E^m\tau_k(n)
= K_{m,k} x + x^{1/\ms} R_{m,k-2}(\log x) + \Omega(x^\alpha),
\qquad
\alpha = {1 \over 2\bigl(\m+1/(k-1)\bigr)}
$$
\end{theorem}
\begin{proof}
By substitution
$$f_1=\cdots=f_k=g_1=\cdots=g_{(k+2)(k-1)/2}=\zeta,$$
$$m_1=1, \quad m_2=\cdots=m_k=\m,$$
$$n_1=\cdots=n_{k-1}=\m+1, \quad n_k=\cdots=n_{(k+2)(k-1)/2} = 2\m$$
into \cite[Th.~2]{kuhleitner1994} we obtain
$$
\alpha = {k-1\over 2(1+\m(k-1))}.
$$
\end{proof}

\section{Multiexponential three-dimensional divisor function}

In the case of $\tau_3$ the statement of Theorem \ref{th:Emtauk-asymp} can be improved
using stronger estimates for $ \theta(1,\m,\m) $. Exponent pairs of form
$$A^{\mms-1} B A^{\mms-2} BABA^2 (k_0, l_0)$$
seems to be especially useful for this task. We shall prove only the simplest result of this kind.

\begin{lemma}\label{l:sum-of-tau1mm-precise}
For a fixed integer $r\ge10$ we have
\begin{equation}\label{eq:sum-of-tau1mm-precise}
\theta(1,2^r,2^r) = {
26\cdot 2^{2r} - (29r+41) 2^r + 16r^2+12r+32
\over
26\cdot 2^{3r} - (16r+41) 2^{2r} + (24r-3) 2^r + 16r+12
} < {1\over 2^r+1}.
\end{equation}
\end{lemma}
\begin{proof}
Follows from the application of \cite[Th. 6.2]{kratzel1988} with
$$
(k,l) = A^{r-1} B A^{r-2} BABA^2 \cdot B (0,1).
$$
Here $(k,l)$ can be evaluated similar to Lemma \ref{l:sum-of-tau1m-precise}:
\begin{align*}
k &= { 13\cdot 2^r - 16r -12 \over 26\cdot 2^{2r} - (16r+54) 2^r + 32r+24 },
\\
l &= 1 - { 13r\cdot 2^r - 16r^2+4r-20 \over 26\cdot 2^{2r} - (16r+54) 2^r + 32r+24 }.
\end{align*}
\end{proof}

For larger $r$ more complicated exponent pairs can be used; they provide slightly lesser values of $\theta(1,2^r,2^r)$.

For small $m$ good estimates of $\theta(1,\m,\m)$ may be obtained by the careful manual selection of exponent pairs, appropriate for the substitution into \cite[Th.~6.2, Th.~6.3]{kratzel1988}. We have calculated several first estimates, see Table \ref{tbl:tau1mm-precise}.

\begin{table}[h]
\begin{tabular}{@{}r|r|r|c@{}}
$m$ & $\m$ & $\theta(1,\m,\m)$  & Exp. pair or ref.\\\hline
0 &     2 & $ 8/25 = 0.320000 $ & \cite[Th. 6.4]{kratzel1988} with $k=3$ \\\hline
1 &     4 & $ (7+\sqrt{809})/190 \approx 0.186542 $ & $ A (BA)^4 (A^2 BA A)^\infty I $ \\\hline
2 &    16 & $ 93607/1698654 \approx 0.055107 $ & $ A^3BA(A(BA)^2A)^3BAA^3(BA)^2 A BA I $ \\
\end{tabular}
\smallskip
\caption{Special values of $\theta(1,\cdot,\cdot)$. Exponent pairs are written in terms of  $A$- and $B$-processes. Here~$I=(0,1)$.}
\label{tbl:tau1mm-precise}
\end{table}

\begin{theorem}
For a fixed integer $m>1$ we have
\begin{equation*}\label{eq:sum-of-Emtau3}
\sum_{n\le x} E^m\tau_3(n) = K_{m,3} x + (r_1\log x + r_0) x^{1/\ms} + \OO(x^{1/(\ms+1)}),
\end{equation*}
where $K_{m,3}$, $r_1$ and $r_0$ are computable constants.
\end{theorem}
\begin{proof}
The statement follows from \eqref{eq:Emtauk-zeta-expansion}, \eqref{eq:sum-of-tau1mm-precise}
and the convolution method.
\end{proof}

Now we are going to refine the last estimate under RH. We need the following lemma, which generalizes \cite[Th. 2]{nowak1993}.

\begin{lemma}\label{l:nowak-generalized}
Consider a multiplicative function $f$ such that
$$
\sum_{n=1}^\infty {f(n) \over n^s}
= {\zeta(as)\zeta^r(bs)\over\zeta^k(cs)} := F(s),
$$
where $2a\le b<c<2(a+b)$.
Let $\Delta(x)$ be defined implicitly by the equation
$$
S(x) := \sum_{n\le x} f(n) = \left( \res_{s=1/a}+\res_{s=1/b} \right) F(s) x^s s^{-1} + \Delta(x).
$$
Abbreviate $\underbrace{b,\ldots,b}_{r \text{~times}}$ as $b\times r$.
Then under RH for any $1\le y \le x^{1/c}$
\begin{equation*}\label{eq:nowak-generalized}
\Delta(x) = \sum_{l\le y} \mu_k(l) \Delta(a,b\times r; x/l^c) + \OO(x^{1/2a+\eps} y^{1/2-c/2a} + x^\eps),
\end{equation*}
where $\mu_k$ is a multiplicative function such that
$$ \sum_{n=1}^\infty {\mu_k(n)\over n^s} = \zeta^{-k}(s) .$$
\end{lemma}
\begin{proof}

For a fixed $y$ let us split $f(n)$ into two parts: $f(n) = f_1(n) \hm+ f_2(n)$, where
$$
f_1(n) = \sum_{l^cm=n \atop l\le y} \mu_k(l) \tau(a,b\times r;m),
\qquad
f_2(n) = \sum_{l^cm=n \atop l> y} \mu_k(l) \tau(a,b\times r;m).
$$
This split naturally implies a split of $S(x)$ into $S_1(x) := \sum_{n\le x} f_1(n)$ and~$S_2(x) \hm{:=} \sum_{n\le x} f_2(n)$. For $S_1$ we obtain
\begin{multline}\label{eq:nowak-s1}
S_1(x) = \sum_{l\le y} \mu_k(l) \sum_{m\le x/l^c} \tau(a,b\times r;m) = \left( \res_{s=1/a}+\res_{s=1/b} \right) F_1(s) x^s s^{-1}
+ \\
+ \sum_{l\le y} \mu_k(l) \Delta(a,b\times r; x/l^c),
\end{multline}
where
$$
F_1(s) := \zeta(as) \zeta^r(bs) \sum_{l\le y} {\mu_k(l) \over l^{cs}}.
$$
Here
\begin{multline*}
\res_{s=1/b} F_1(s) x^s s^{-1}
= x^{1/b} \sum_{l\le y} {\mu_k(l) \over l^{c/b}}
P(\log x)
- c x^{1/b} \sum_{l\le y} {\mu_k(l) \log l \over l^{c/b}} \cdot {1\over b} \zeta(a/b),
\end{multline*}
where $P$ is a polynomial with $\deg P = r-1$.
As soon as $\sum_{l\le x} \mu_k(l) \hm\ll x^{1/2+\eps}$ under RH we get
$$
x^{1/b+\eps} \sum_{l>y} \mu_k(l) l^{-c/b+\eps}
\ll x^{1/b+\eps} y^{1/2-c/b+\eps}
\ll x^{1/2a+\eps} y^{1/2-c/2a}.
$$
So
\begin{equation}\label{eq:nowak-res-f1}
\res_{s=1/b} F_1(s) x^s s^{-1}
= \res_{s=1/b} {\zeta(as)\zeta^r(bs)\over\zeta^k(cs)}
+ \OO(x^{1/2a+\eps} y^{1/2-c/2a}).
\end{equation}

Now consider $S_2$. Define
$$
g(s) := \sum_{l>y} {\mu_k(l) \over l^s} \qquad \text{for~} \sigma>1.
$$
Then under RH function $g$ can be continued analytically to $\sigma > 1/2+\eps$ and we have uniformly for all such $\sigma$ that
$$ g(s) \ll y^{1/2-\sigma+\eps} \bigl(|t|+1\bigr)^\eps $$
(see \cite[L. 3]{nowak1993} for the proof). Thus by Perron formula with $c=1/a+\eps$ and~$T=x^2$ we get
$$
S_2(x) = \int_{c-iT}^{c+iT} F_2(s) x^s s^{-1} \dd s + \OO(x^\eps),
$$
where $F_2(s) := \zeta(as) \zeta^r(bs) g(cs)$. Moving the line of integration to~$\sigma \hm= d \hm= 1/2a+\eps$ we obtain
\begin{multline}\label{eq:nowak-s2}
S_2(x) - \res_{s=1/a} F_2(s) x^s s^{-1}
\ll
\left( \int_{d+iT}^{c+iT} + \int_{d-iT}^{d+iT} + \int_{d-iT}^{d+iT} \right) F_2(s) x^s s^{-1} \dd s
=: \\
=: I_1 + I_2 + I_3.
\end{multline}
But for $\sigma > 1/2a$
$$
g(cs) \ll y^{1/2-c/2a+\eps} \bigl(|t|+1\bigr)^\eps.
$$
And RH implies Lindelöf hypothesis, so $\zeta(s) \ll \bigl(|t|+1\bigr)^\eps$ for $\sigma>1/2$. Thus
\begin{equation}\label{eq:nowak-i123}
I_1+I_2+I_3 \ll y^{1/2-c/2a+\eps} x^\eps.
\end{equation}
Combining of \eqref{eq:nowak-s1}, \eqref{eq:nowak-res-f1}, \eqref{eq:nowak-s2} and \eqref{eq:nowak-i123} completes the proof.
\end{proof}

\begin{theorem}
Under RH for a fixed $m>1$ we have
\begin{equation*}
\alpha_{m,3}  = {1-\theta(1,\m,\m) \over \m+2-2(\m+1)\theta(1,\m,\m)},
\end{equation*}
where $\alpha_{m,3}$ is an exponent in \eqref{eq:Emtauk-estimate}.
\end{theorem}
\begin{proof}
Similar to Theorem \ref{th:Emtau-under-RH} with the use of Lemma \ref{l:nowak-generalized} with $r=2$ instead of Nowak's result. Equation \eqref{eq:nowak-result} transforms into
\begin{equation*}\label{eq:nowak-threeD-result}
\sum_{n\le x} (\tau(1,b,b; \cdot)\star \mu_c)(n)
= C_1 x^{1/a}
+ (C_2\log x+C_3) x^{1/b}
+ \OO(x^{\beta+\eps}),
\end{equation*}
with
\begin{equation}\label{eq:nowak-threeD-exponent}
\beta = {1-a\theta(a,b,b) \over a+c-2ac\theta(a,b,b)}.
\end{equation}
\end{proof}

We give one more application of Lemma \ref{l:nowak-generalized} on somewhat off-topic function. Define~$\tt\colon \Z[i]\to\Z$, where $\tt(n)$ is a number of $d\in\Z[i]$ such that $d \mid n$ over~$\Z[i]$. Consider~$\taue_* := E(\tt|_{\Z})\colon \Z\to\Z$. We proved in \cite[(10)]{lelechenko2013-ntsh} that
$$
\sum_{n=1}^\infty {\taue_*(n) \over n^s}
= {\zeta(s) \zeta^2(2s) \over \zeta(3s)} H(s),
$$
where $H(s)$ converges absolutely for $\sigma > 1/5$, and due to \cite[Th. 3]{lelechenko2013-ntsh}
$$
\sum_{n\le x} \taue_*(n) =
A_1 x + (A_2 \log x + A_3) x^{1/2} + \OO(x^{1/3+\eps}).
$$
Lemma \ref{l:nowak-generalized} shows that the last estimate can be improved under RH. Taking into account \eqref{eq:nowak-threeD-exponent} with $(a,b,c)=(1,2,3)$, we obtain for $\theta(1,2,2) \hm= 8/25$ that
$$
\beta = {1-\theta(1,2,2) \over 4-6\theta(1,2,2)}
= {17\over52} < 1/3.
$$

\section{Generalization}

Consider an arbitrary multiplicative function $f$. Define
$$
A(f) := \{ n \mid f(n)\ne1 \},
\qquad
n(f) := \min A(f),
$$ $$
B(f) := A(f) \setminus (2\N+1) n(f),
\qquad
n'(f) := \min (A(f) \setminus \{n(f)\} ).
$$
Then
$n(E^m f) = 2^{n(E^{m-1} f)}.$

\begin{lemma}\label{l:lowest-arguments}
Let $f$ be an arbitrary arithmetic function. For every $k\in\N$ there exists sufficiently large, constructively defined $m_0 = m_0(k,f)$ such that for $m\ge m_0$ integers
$$
n(E^m f), 3 n(E^m f), \ldots, (2k-1) n(E^m f)
$$
are $k$ lowest elements of $A(E^m f)$.
\end{lemma}
\begin{proof}
Without loss of generality we can suppose that $f$ is multiplicative prime-independent and~$n(f)$ is a power of 2. On contrary we take~$Ef$ instead of $f$.

We can also suppose without loss of generality that $n'(f) \ge 2 n(f)$. On contrary we can take $Ef$ instead of $f$ once again, because
$n'(f) \hm\ge n(f)+1$ and
$$
n'(Ef) \ge \min \left\{ 3 n(Ef), \min B(Ef) \right\},
$$
where
$$
\min B(Ef) \ge \min \left\{ 3^{n(f)}, 2^{n'(f)} \right\} \ge \min \left\{ (3/2)^2 \cdot 2^{n(f)}, 2^{n(f)+1}  \right\} = 2 n(Ef).
$$

Surely if $n'(f) \ge 2 n(f)$ then for every $m$ we have $n'(E^m f) \hm\ge 2 n(E^m f)$ too.

\smallskip

Now to prove the statement of the lemma is enough to show that
$$\min B(E^m f) \ge 2 k n(E^m f).$$
As soon as sequence $n(f), n(Ef), n(E^2f)\ldots$ is tending to $+\infty$, we can choose $m$ such that
$
n(E^m f) \ge 2k $
 and
$(3/2)^{n(E^{m-1}f)} \ge 2k
$.
Then
\begin{multline*}
\min B(E^m f) \ge \min\left\{ 3^{n(E^{m-1}f)}, 2^{n'(E^{m-1}f)} \right\}
\ge \\ \ge
\min \left\{ 2kn(E^m f), \bigl( n(E^mf) \bigr)^2 \right\} \ge 2kn(E^m f).
\end{multline*}

\end{proof}

\begin{lemma}
Let $f$ be an arbitrary arithmetic function and let $m_0 \hm= m_0(2,f)$ be as in Lemma \ref{l:lowest-arguments}. For a fixed integer $m\ge m_0$ we have
$$
\sum_{n=1}^\infty {E^mf(n) \over n^s} = \zeta(s)
\left( { \zeta\bigl(n(E^mf) s\bigr) \over \zeta\Bigl( \bigl( n(E^mf)+1 \bigr) s \Bigr) } \right)^{f(n(f))-1} H(s),
$$
where $H(s)$ converges absolutely for
$
\sigma > 1/\bigl( 2n(E^m f) \bigr)
$.
\end{lemma}
\begin{proof}
By Lemma \ref{l:lowest-arguments} all values $E^mf(1),\ldots,E^mf\bigl(3n(E^mf)-1\bigr)$ equals to 1 except of $E^mf\bigl(n(E^mf)\bigr) = f\bigl(n(f)\bigr)$. Thus
$$
S(x) := \sum_{n\ge0} E^mf(p^n) x^n
= {1\over1-x} + \Bigl(f\bigl(n(f)\bigr)-1\Bigr) x^{n(E^mf)} + \OO\left(x^{3n(E^mf)}\right),
$$
$$
(1-x)
\left( {
1-x^{n(E^mf)} \over 1-x^{n(E^mf)+1}
} \right)^{f(n(f))-1}
S(x)
= 1 + \OO\left(x^{2n(E^mf)}\right).
$$
This asymptotic identity implies lemma's statement.
\end{proof}

Now if $f\bigl(n(f)\bigr) \in \N$ then Theorem \ref{th:toth-improved} with
$
k = f\bigl(n(f)\bigr)$,
$l = n(E^mf)
$
can be applied on $E^m f$.

\bibliographystyle{ugost2008s}
\bibliography{taue}

\end{document}